\PassOptionsToPackage{x11names}{xcolor}
\documentclass[12pt,a4paper,twoside]{article}

\usepackage{array}
\usepackage{amsmath}
\usepackage{tabularx}
\usepackage{txfonts}
\usepackage{enumerate}
\usepackage{cite}
\usepackage{float}
\usepackage{xspace}
\usepackage[x11names]{xcolor}
\usepackage{multirow}

\usepackage{lineno}

\usepackage{amssymb}

\usepackage[T1]{fontenc}
\usepackage[cp1250]{inputenc}

\usepackage[margin=2cm]{geometry}

\usepackage{changes}
\usepackage[colorinlistoftodos, textwidth=30mm, shadow]{todonotes}

\newtheorem{theorem}{Theorem}[section]

\newenvironment{proof}[1][]{\par\noindent\textbf{Proof #1: }}{\hfill\rule{1.3ex}{1.3ex}\\\par}

\begin{document}

\pagestyle{myheadings}
\markright{G{\'o}recka et al.: Description of the emotional states of communicating people by mathematical model}

\noindent\begin{tabular}{|p{\textwidth}}
	\Large\bf Description of the emotional states of communicating people by mathematical model\\\vspace{0.01cm}
    \it Joanna G{\'o}recka$^1$,\\
    \it Urszula Fory{\'s}$^2$,\\
    \it Monika Joanna Piotrowska$^{2,*}$\\\vspace{0.02cm}
\it\small $^1$College of Inter-Faculty Individual Studies in Mathematics and Natural Sciences,\\
\it\small University of Warsaw, \.Zwirki i~Wigury 93, 02-089 Warsaw, Poland\\\vspace{0.01cm}
\it\small $^2$Institute of Applied Mathematics and Mechanics,\\
\it\small Faculty of Mathematics, Informatics and Mechanics,\\
\it\small University of Warsaw, Banacha 2, 02-097 Warsaw, Poland\\\vspace{0.01cm}
\footnotesize  \texttt{$^*$ monika@mimuw.edu.pl},\\
    \multicolumn{1}{|r}{\large\color{orange} Research Article} \\
	\\
	\hline
\end{tabular}
\thispagestyle{empty}

\tableofcontents
\noindent\begin{tabular}{p{\textwidth}}
	\\
	\hline
\end{tabular}

\begin{abstract}
The model we study is a generalization of the model considered by Liebovitch~\textit{et al.}~(2008) and Rinaldi~\textit{et al.}~(2010), and is related to the discrete model of the emotional state of communicating couples described by Gottman~\textit{et al.}~(2002). Considered system of non-linear differential equations assumes that the emotional state of a person at any time is affected by the state of each actor alone, rate of return to that state, partner's emotional state and mutual sympathy. Interpreting the results, we focus on the analysis of the impact of a person's attitude to life (optimism or pessimism) on establishing emotional relations. It occurs that our conclusions are not always obvious from the psychological point of view.
\end{abstract}








\section{Introduction}\label{sec:intro}

People are social beings by nature. Almost all our life we talk,  work or play with someone; contacts with other 
persons have a~strong influence on our emotions. This impact can be positive or negative depending on various 
factors and is most visible in the relationship between partners. People can like each other, and then a~good 
mood of one person positively affects the mood of the other one, while a~bad mood has a~negative effect. Two 
persons may also not like each other, and then an emotional impact of one partner to the other is opposite. 
Hence, a~sadness of one person improves the mood of the other one and \textit{vice versa}. The last 
possibility is that the relationship is mixed, and one person has a  negative attitude to the second one, who is 
geared to their friendship. In such a~relationship the emotional impact of one partner to the other is a~compound of previously described influences.

If we assume that our emotional state is influenced only by other people, then the obvious strategy is to be friendly for all the close-knit people and to avoid relationships with people who have negative attitude to us. In general, we hope to be like that, however if that is true, then everyone would be friendly and happy, and it is not the case. 
  Although one can try to explain that effect by saying that the feeling of pleasure is not 
the only desire of a~man and different random events can destroy the balance, it would still not represent the 
full picture. 

On the other hand, the individual factors also have an impact on our emotional state. 
These are certainly: the attitude to life (optimism and pessimism) and how strong is the influence of the current 
mood of a~person on the change of her/his emotional state.

To analyse the impact of the level of optimism on the profits from the meeting of people with different attitude we decided to base on a~simple mathematical model considered earlier by Liebovitch {\it et al.}~\cite{Leibovitch2008} and Rinaldi {\it et al.}~\cite{Rinaldi2010}. In this model the impact of the factors  mentioned above on the change of emotional states of two considered persons was described. We study this model from a~different point of view than Liebovitch {\it et al.} or Rinaldi {\it et al}. Moreover, we introduce some modifications to the model, which  in our opinion allow to better reflect real relationships.

\section{Model Description}\label{sec:mod}

Many models describing interactions between people follow the idea of marital interactions described 
in~\cite{Gottman2002, Murray2002}. Such models can even be linear, as considered 
in~\cite{Strogatz1988, Strogatz1994, Felmlee1999, Rinaldi1998_2, Bielczyk2012, nbuftp11jms,nbmbufjp10kkzmbm}. However, nonlinear 
models seem to be more appropriate. 
 One of such models is the model of emotional states of communicating people considered 
in~\cite{Leibovitch2008}.  
In this paper we base on the Liebovitch {\it et al.} model~\cite{Leibovitch2008}, proposing some changes and 
a~new interpretation. It should be marked here, that the model of the same structure was also considered by Rinaldi and his co-authors in the series of papers~\cite{Rinaldi2010, Rinaldi1998_1, Rinaldi2013, Rinaldi2014, Rinaldi2015}, however our interpretation is more close to those given in~\cite{Leibovitch2008}.

\subsection{Model history}

Before we present the model considered in this paper, we shall say few words about its roots. The prototype of 
the Leibovith {\it et al.} model is the model of marital interactions created by Gottman and Murray, which 
directly comes from empirical research~\cite{Gottman2002}. At the end of the 20th century Gottman 
conducted an experiment in which married couples with problems possibly leading to dissolution took part. In his clinic, Gottman 
 observed 15-minute conversations each of the 73 couples on a difficult 
subject. Wife and husband alternately spoken, and each positive and negative communication (verbal or nonverbal) 
during the conversation was recorded. This gave an observational code of interactive behaviour called RCISS (Rapid Couples Interaction Scoring System).  Finally, for each couple two series of data were obtained and reflected in a graph in which the differences of positive and 
negative messages of wife and husband in every ``round'' were marked. According to the 
experimental data the system of equations, which model this difference, was postulated
\begin{equation}\label{sys:dys}
\begin{split}
W_{t+1}&=I_{HW}(H_t)+r_1W_t+a,\\
H_{t+1}&=I_{WH}(W_{t+1})+r_2H_t+b,
\end{split}
\end{equation}
where $W_t$, $H_t$ are the scores of wife and husband in round $t$, respectively, constants $a$, $b$ and $r_i$ ($i=1,2$) 
determine the rate at which individual returns to independent steady state, and $I_{AB}
(A_t)$ is a function of the impact of person $A$ on $B$ in the round $t$. 
The equations of system~\eqref{sys:dys} are non-symmetric because the wife talked first in each round. 
Discrete-time model was used because it can easily reflect the experimental data.  
 On the  other hand, one can also construct an analogous continuous-time model, as it was noticed by Murray~\cite{Murray2002}. Such a model, 
 describing the emotional state of communicating people, could also be used 
in more general situation. Simply one needs to assume that happy person sends positive signals,  while the 
negative signals are sent by unhappy person.  Based on the model reflected by~\eqref{sys:dys}  Leibovith~{\it et al.} 
proposed a continuous model, which we modify and analyse  in detail in this paper from different perspective.

\subsection{Presentation of the interaction model}

In this subsection we present continuous-time model on which our analysis is based. The model of such a structure  was earlier considered by   Liebovitch~{\it et al.}~\cite{Leibovitch2008} and also by Rinaldi {\it et al.}~\cite{Rinaldi2010}. However, interpretation of the model parameters is different in the papers of Rinaldi and co-authors~\cite{Rinaldi2010, Rinaldi1998_2, Rinaldi1998_1, Rinaldi2013, Rinaldi2014, Rinaldi2015}. 

Let $x(t)$ and $y(t)$ reflect the emotional states of two distinguishable individuals at time $t$. We consider 
the system of differential equations that reads
\begin{equation}
\label{podst_ukl_2}
	\begin{aligned}
		\dot x(t) &= -m_{1}x(t) + b_{1} + c_{1}f_{1}(y(t)), \\
		\dot y(t) &= -m_{2}y(t) + b_{2} + c_{2}f_{2}(x(t)),
	\end{aligned}
\end{equation}
where constants $ m_ {1} $ and $ m_ {2} $ describe the rate of change of the mood of each 
person in solitude, which can be also referred as to forgetting coefficients, 
$ b_ {1} $ and $ b_ {2} $ reflect some ``ideal/reference'' mood of each person, functions $f_{1}$ and $f_2$ 
describe the impact of the emotional state of a~person $y$ or $x$, respectively, on the emotional state of the 
other person, while constants $ c_ {1} $ and $ c_ {2} $ determine the strength and direction of these 
influences. If there is no such as influence, that is for $ c_ {1} = c_ {2} = 0 $, and $m_i>0$, similarly as hypothesised in~\cite{Gottman2002}, then system~\eqref{podst_ukl_2} returns to the steady state $ ( \tfrac {b_ {1}} {m_ {1}},\tfrac {b_ {2}} {m_ {2}}) $, where the  state $ \tfrac {b_ {i}} {m_ {i}}$ is called uninfluenced equilibrium for person $i$ in~\cite{Gottman2002}. 
 When both $c_1$ and $c_2$ are positive, people have a~positive attitude to each other, while for 
$ c_ {1}$, $c_ {2} 
<0 $ they have a~negative attitude to each other. Clearly, for $ c_ {1}> 0 $ and $ c_ {2} <0 $ 
a~person $ x 
$ has a~positive attitude towards $ y $ and $ y $ has a~negative attitude to $ x $.
As we describe interactions between two persons, we assume $c_1\cdot c_2 \ne 0$. 
Moreover, following the ideas of Gottman {\it et al.} we assume that 
\[ m_ {1}, \,m_ {2} >0,\]
and thus the person being in 
isolation, not influenced by other person,
approaches his/her steady state $\tfrac{b_i}{m_i}$. 

Under the assumption above, the person characterised by positive parameter $ b_ {i} $ has a positive steady state and is called an 
optimists, while those with negative   parameter -- a pessimists. 
Rinaldi {\it et al.}~\cite{Rinaldi2010} gave completely different interpretation of this parameter. In their interpretation $b_1$ reflects appeal of the person $y$ for $x$, so when the person is in solitude, this parameter is just equal to $0$, as well as uninfluenced steady state, because there is no love/hate whenever there is no object of these emotions. It should be marked that although Rinaldi got interesting results using this interpretation (e.g. he was able to explain the case of Beauty and the Beast~\cite{Rinaldi2013}), 
 we shall not follow this idea, but use the interpretation of Gottman~\cite{Gottman2002} and then Liebovitch~\cite{Leibovitch2008}. 

Various particular influence functions $f_i$ were considered in the literature, for details see {\it e.g.} 
\cite{Leibovitch2008, Gottman2002, Murray2002, Rinaldi2013}. However, in this paper we propose $f_i$ in more general 
form based on the prospect theory of decision making problems.

We should also marked that under our interpretation the model described by~\eqref{podst_ukl_2} reflects the emotional state of people during a~single meeting with 
a~partner, but not, for example, a~series of meetings. This is because between two meetings people meet other 
people or spend time in solitude, which affects their mood at the beginning of the meeting, and therefore also 
the final result.

\subsection{Influence functions}
The prospect theory proposed by Kahneman and Tversky~\cite{Kahneman1979} relates to the wider issues of risk 
assessment and an attitude of men to the risk. Here, we briefly introduce the assumptions, which are useful 
from our point of view. There are three main principles of profit and loss assessment by people. First, generally 
we experience losses much stronger than the profits of the same value. Second, everyone defines their own 
criteria with respect to which results of the decision are evaluated as a~gain or loss. Third, every unit of gain is 
enjoyed with diminishing efficiency, and each consecutive loss is less saddening.

Although the prospect theory describes the relationship between profits or losses and satisfaction, we believe 
that it can be used to describe the mutual influence of partners' emotional states. When we are alone our 
emotional state depends only on our character. When we meet a~friend who is happy, we gain his emotions, 
but not in the literal sense. We react to his/her smile, lively tone of voice, {\it etc}. When we get a~more 
positive stimulus we enlarge our profit more. However, according to the prospect theory, each additional unit 
gives less and less profit. Therefore, influence functions $f_i$ are certainly non-linear. The first derivative of it 
should be positive, decreasing for positive variables and increasing for negative ones. Moreover, it should tend to 
0 in $\pm \infty$, because otherwise $f_i$ become almost linear asymptotically. 

Another issue is that generally we feel losses much stronger than profits of the same value. Indeed, people 
recognise negative emotions faster, more accurately and more strongly than positive ones. It is an adaptive 
process, because when we talk about surviving, the ability to recognise and quickly respond to the feeling of fear or anger is more important than joy. Hence, the negative experience makes us more sad than the same weight positive 
experience makes us happy.

Last feature of that theory stating that everyone assesses gains and losses from his/her own point of view, actually, 
is not so important from our model point of view. Benchmark is always a~state in solitude, that means 
a~situation in which the environmental impact is equal to zero. However, if we do not like someone, then his 
negative emotions are a~profit for us, and the positive emotions are our loss,  while if we like someone, it is 
\textit{vice versa}.



Concluding, to address the issues described above we assume: 

\[
f_{i}\in \mathbf{C}^{2}, \quad f_{i}(0) =  0, \quad f_{i}'(\xi)>0,
\]
and
\[ 
\lim\limits_{|\xi| \to \infty}f_i'(\xi)=0 , \quad \xi f_{i}''(\xi)< 0 \ \ \text{for} \ \ \xi \ne 0,  \quad \textrm{for}  \quad
 i=1, \ 2.
\]

\noindent
 Moreover, to more explicitly show that the force of impact of partners on themselves depends on the value of 
$ | c_ {i} | $   we also assume that 
\[f_{i}'(0) = 1, \quad \textrm{for}\quad i=1, \ 2.\]

Exemplary graphs of influence functions are shown in Fig.~\ref{wp}. Arrows indicate how to read the graphs. 

\begin{figure}[!ht]
\centering\includegraphics[height=5cm]{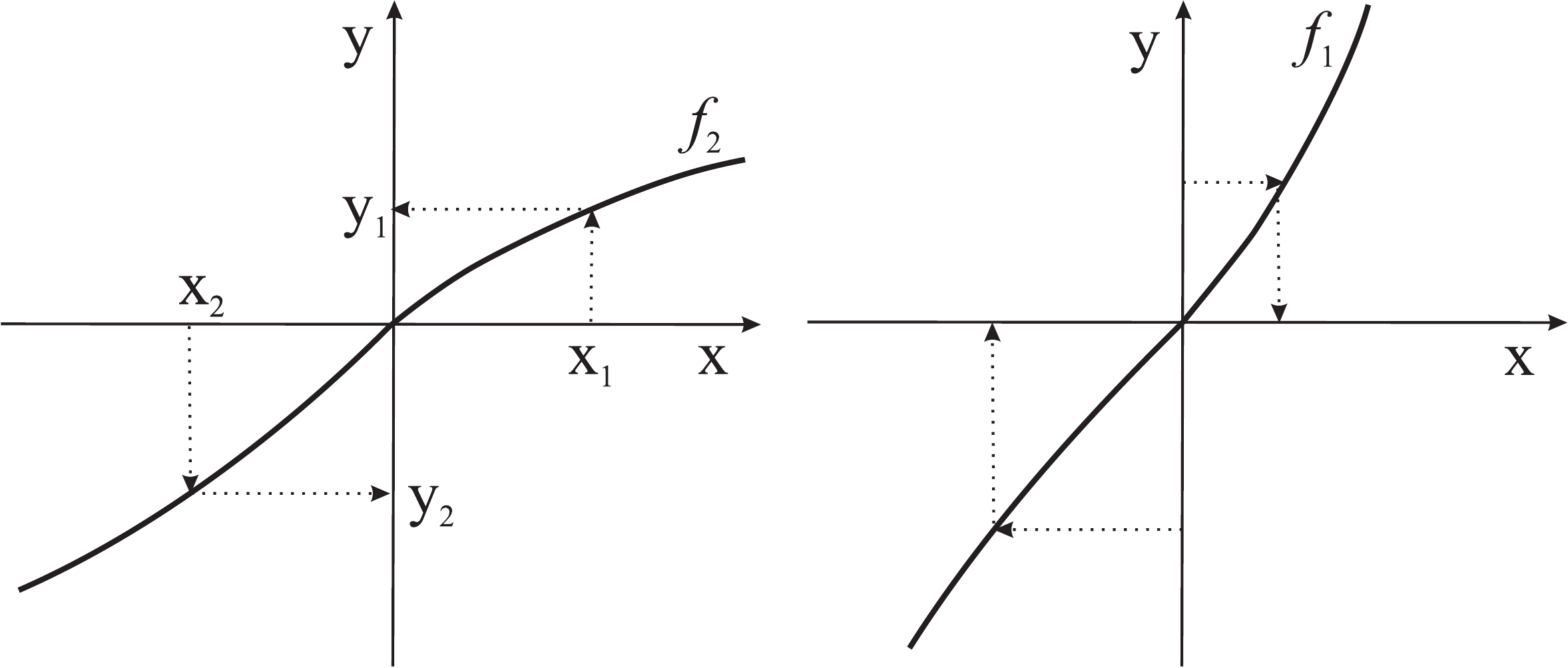}
\caption{Examples of 
the influence function.}\label{wp}
\end{figure}

It is worth to notice that functions $ f_i $ may have a~different shape because people differ in ability to 
distinguish emotions, and consequently -- reactions to them. This means that $ f_ {1} '(y) $ and $ f_ {2}' (x) $ 
can change with different rates.
Clearly, proposed interaction functions are defined in the general form, and hence we actually consider the 
whole family of functions. On the other hand, the interaction functions considered previously in the literature 
belong to this family.

\section{Model analysis}\label{sec:an}

Under our assumptions, it is obvious that unique solutions of Eqs.~\eqref{podst_ukl_2} exist, and moreover due to properties of the functions $f_i$, any solution can be  
prolonged on the whole interval $[0, \infty)$. These basic properties do not depend on initial data, the model parameters and forms of $f_i$. Clearly, without particular forms of the influence functions we 
are not able to determine steady states explicitly. Hence, assume $(x_{s}, y_{s})$ to be a~solution of the 
system 
\begin{equation}\label{ss}
	\left\{
	\begin{aligned}
		x = \frac{ b_{1} + c_{1}f_{1}(y)}{m_{1}}, \\
		y = \frac{ b_{2} + c_{2}f_{2}(x)}{m_{2}},
	\end{aligned}
	\right.
\end{equation}
{\it i.e.} $(x_s,y_s)$ is a~steady state. 
The first equation of~\eqref{ss} describes null-cline $I_1$ for the first variable $x$, while the second one -- 
$I_2$ for $y$. Positions of $I_1$ and $I_2$ in the phase space $(x,y)$ depend on the model parameters and specific forms of $f_i$. In 
the analysis presented below we treat both null-clines as functions $g_i(x)$. Therefore, $g_2$ is defined for all 
$x \in \mathbb{R}$ with $g_2'(x) \to 0$ as $|x| \to \infty$, while $g_1$ takes all values from $\mathbb{R}$  
and $g_1'(x) \to \infty$ as $x$ tends to the end of its domain (either $\mathbb{R}$ or some bounded 
interval). 

 Clearly, if $c_1c_2<0$, then one of the null-clines is increasing, while the other is decreasing, and  therefore they 
intersect at exactly one point. If $c_1\cdot c_2>0$, then Eqs.~\eqref{ss} has always from one to three 
solutions depending on the parameter values. We briefly discuss the case $c_1$, $c_2>0$, as the case with 
negative parameters is symmetric. Consider first    $b_1=b_2=0$. Then null-clines always intersect at $x=0$ and 
other intersection points appear when $g_2'(0)>g_1'(0)$, that is $c_2/m_2>m_1/c_1$, while for $g_2'(0)\leq 
g_1'(0)$  there is only one intersection.
For $g_2'(0)\leq g_1'(0)$   shifting null-clines  does not change the situation, while for $g_2'(0)> g_1'(0)$, if we 
shift at least one of the null-clines sufficiently far from the origin, then additional intersection points disappear. 
\begin{enumerate}
\item If 
\begin{itemize}
\item either $c_{1}c_{2} \leq m_{1}m_{2} $, 
\item or $c_{1}c_{2} > m_{1}m_{2}$ and there is no solution of Eqs.~\eqref{ss} satisfying 
$f_{1}'(y_{s})f_{2}'(x_{s}) \geq \frac{m_{1}m_{2}}{c_{1}c_{2}}$,
\end{itemize}
 then System~\eqref{podst_ukl_2} has one steady state.
\item If $c_{1}c_{2} > m_{1}m_{2}$, $b_{1} \neq 0$ or $b_{2} \neq 0$, and solutions of Eqs.~\eqref{ss} fulfil 
the conditions $f_{1}'(y_{1})f_{2}'(x_{1}) = \frac{m_{1}m_{2}}{c_{1}c_{2}}$ and 
$f_{1}'(y_{2})f_{2}'(x_{2}) < \frac{m_{1}m_{2}}{c_{1}c_{2}}$, then 
System~\eqref{podst_ukl_2} has two steady states.
\item If $c_{1}c_{2} > m_{1}m_{2}$ and there exists a solution of Eqs.~\eqref{ss} fulfilling  
$f_{1}'(y_{s})f_{2}'(x_{s}) > \frac{m_{1}m_{2}}{c_{1}c_{2}}$, then 
System~\eqref{podst_ukl_2} has three steady states. Two additional steady states $(x_{i},y_{i})$, $i \in \{1,2\}$ 
fulfil $f_{1}'(y_{i})f_{2}'(x_{i}) < \frac{m_{1}m_{2}}{c_{1}c_{2}}$.
\end{enumerate}

To study the stability of steady states we calculate the Jacobi matrix at a~steady state $ (x_ {s}, y_ {s}) $,
\begin{equation*}
	J(x_{s}, y_{s}) = \left(
		\begin{array}{cc}
		-m_{1} & c_{1}f_{1}'(y_{s}) \\
		c_{2}f_{2}'(x_{s}) & -m_{2}
		\end{array}
	\right),
\end{equation*}
and the corresponding characteristic polynomial 
\[
W(\lambda) = \lambda^{2} - A\lambda + B, 
\]
with 
\[ 
A= -(m_{1} + m_{2})<0, \quad B = 
m_{1}m_{2} - c_{1}c_{2}f_{2}'(x_s)f_{1}'(y_s).
\]
 Consequently, eigenvalues are equal to $\lambda_{1,2} = \frac{A~\pm \sqrt{A^{2} - 4B}}{2}$. 

To have locally asymptotically stable steady state $ (x_ {s}, y_ {s}) $ we  look  for  eigenvalues with negative real 
parts. From the assumptions of the model ($ m_ {1}$, $m_ {2} >0 $) we have $ A~<0 $. 
Therefore, stability depends only on the value of $B$. We see that if $ B $ is positive, then the 
steady state $ (x_ {s}, y_ {s}) $ is stable, and it is either focus (for $A^2-4B<0$) or node (for $A^2-4B>0$), while if $B$ is negative, $(x_s,y_s)$ is a saddle. 
Moreover, whenever there exists the unique steady state $(x_s,y_s)$ and the impact of partners is not higher than the inner dynamics of the interacting  persons, then $(x_s,y_s)$  is globally stable, while if there are three steady states, then the dynamics changes to bi-stable. In the case of partners having opposite attitude  to each other the steady state is unique, and moreover it is globally stable independently of other model parameters.

More precisely, for the case with unique steady state, and we are able to prove the following theorems 
\begin{theorem}\label{th:1}
If $m_1m_2>|c_1c_2|$ and System~\eqref{podst_ukl_2} has exactly one steady state $(x_s,y_s)$ satisfying Eqs.~\eqref{ss}, then this state is globally stable. 
\end{theorem}
\begin{proof}
Proving global stability we use the method of Lyapunov functions. Let us define
\[
V(x,y)=\frac{1}{2}\left( x-x_s \right)^2+ \frac{C}{2}\left( y-y_s \right)^2 ,
\]
where $C>0$ is a constant, which should be chosen in appropriate way.
Calculating the derivative along trajectories of System~(2)
we get
\[
\frac{dV}{dt}=\left( x-x_s \right)\left(-m_1x+b_1+c_1 f_1(y)\right)+
C \left( y-y_s \right)\left(-m_2y+b_2+c_2 f_2(x)\right) .
\]
Using the relations $b_1=m_1 x_s -c_1 f_1(y_s)$, $b_2=m_2 y_s -c_2 f_2(x_s)$, and the mean value theorem  we obtain
\[
\frac{dV}{dt}= -\Big( m_1 \left(x-x_s\right)^2 +C m_2\left(y-y_s\right)^2 - \left(c_1f_1'(y_p) +C c_2f_2'(x_p)\right)\left(x-x_s\right)\left(y-y_s\right)
\Big) ,
\]
where $x_p$ and $y_p$ are the points between $x$, $x_s$ and $y$, $y_s$, respectively, and the right-hand side could be treated as a quadratic form of $x-x_s$ and $y-y_s$. Hence, we need to study positivity of the matrix
\[
M=\left( \begin{matrix}
m_1 & - \frac{1}{2}\left(c_1f_1'(y_p) +C c_2f_2'(x_p)\right) \\
- \frac{1}{2}\left(c_1f_1'(y_p) +C c_2f_2'(x_p)\right) & C m_2
\end{matrix}\right) .
\]
The matrix $M$ is positive under the assumptions $m_1>0$, $\det M>0$. The first assumption is always satisfied for our system, while the second one is equivalent to the following inequality
\[
C^2c_2^2(f_2'(x_p))^2 +2 C\left( c_1c_2f_1'(y_p)f_2'(x_p) - 2m_1m_2\right)
+ c_1^2(f_1'(y_p))^2  <0 .
\]
As $0<f_1'f_2'\leq 1$ under our assumptions, it is enough to choose $C>0$ such that
\[
C^2c_2^2 -2 C\left(   2m_1m_2-|c_1c_2|\right)
+ c_1^2  <0 ,
\]
which under the assumption $m_1m_2>|c_1c_2|>|c_1c_2|/2$ has real positive solutions, and we can choose $C=\tfrac{2m_1m_2-|c_1c_2|}{c_2^2}>0$,  which gives minimum of the quadratic function above. Therefore, the function $V$ satisfies all assumptions guaranteeing the global stability of $(x_s,y_s)$.
 \end{proof}

\begin{theorem}\label{th:2}
If $c_1c_2<0$, then System~\eqref{podst_ukl_2} has unique steady state $(x_s,y_s)$ satisfying Eqs.~\eqref{ss}, which is globally stable.
\end{theorem}
\begin{proof}
Uniqueness of $(x_s,y_s)$ is obvious. Proving global stability we again use the method of Lyapunov functions. We start with changing variables of System~(2)
such that the steady state $(x_s,y_s)$ is shifted to $(0,0)$. We define $u=x-x_s$ and $v=y-y_s$ for which we obtain
\begin{equation}\label{podst_ukl_shift}
	\left\{
	\begin{aligned}
		\dot u &= -m_{1}u +  c_{1}\left( f_{1}(v+y_s)-f_1(y_s)\right) , \\
		\dot v &= -m_{2}v +  c_{2}\left( f_{2}(u+x_s)-f_2(x_s)\right),
	\end{aligned}
	\right.
\end{equation}
due to Eqs.~\eqref{ss}. Let us consider
\[
L(u,v)=|c_1|\int\limits_0^v \left(f_1(\xi+y_s)-f_1(y_s)  \right)d\xi
+ |c_2|\int\limits_0^u \left(f_2(\xi+x_s)-f_2(x_s)  \right)d\xi .
\]
Because $f_i$ are increasing functions, we have $L(u,v)=0$ iff $u=v=0$ and $L(u,v) \to \infty$ for $u\to \infty$ or $v \to \infty$. Calculating the derivative of $L$ along trajectories of System~\eqref{podst_ukl_shift} 
and using the relation $|c_1|c_2=-|c_2|c_1$
we obtain
\[
\dot L(u,v)=-m_2 |c_1| v \left(f_1(v+y_s)-f_1(y_s)\right) 
-m_1 |c_2| u \left(f_2(u+x_s)-f_2(x_s)\right) .
\]
Now, using the mean value theorem we have 
\[
\frac{f_1(v+y_s)-f_1(y_s)}{v}=f_1'(v_p) \quad  \text{and} \quad  \frac{f_2(u+x_s)-f_2(x_s)}{u}=f_2'(u_p) ,
\]
 where $v_p$ and $u_p$ are intermediate points. Hence,
\[
\dot L(u,v)=-m_2 |c_1| v^2 f_1'(v_p)
-m_1 |c_2| u^2 f_2'(u_p) ,
\]
yielding $\dot L(u,v) \leq 0$ and $\dot L(u,v)=0$ iff $u=v=0$. Therefore, $L$ is a Lyapunov function for System~\eqref{podst_ukl_shift} and $(0,0)$ is globally stable. This proves global stability of $(x_s,y_s)$ for System~(2).
\end{proof}

In general, we have also the following property of System~\eqref{podst_ukl_2}, which is independent of the model parameters.

\begin{theorem}\label{th:okres} There is no periodic solutions of System~\eqref{podst_ukl_2}.
\end{theorem}
\begin{proof}
We use  Dulac--Bendixon Criterion. Let us define $B(x,y)=1$ for all $x$ and $y$. Then  the divergence of the vector field fulfils
\[\text{div} B\left(\frac{dx}{dt},\frac{dy}{dt}\right) =-\left( m_{1} + m_{2}\right)=A<0 .\]
 Thus, System~(2)
has no periodic solutions. 
\end{proof}

Clearly, Theorem~\ref{th:okres} implies that there are not limit cycles of System~\eqref{podst_ukl_2}. Moreover, whenever there is only one steady state and solutions remain in bounded regions in the phase space, then Poincar\'{e}--Bendixson Theorem yields the global stability of this state.

\subsection{Specific types of the model dynamics}\label{sec:res}

Clearly, whenever two people of the opposite attitude to each other meet, there is only one steady state. For 
$ f_{2}'(x_{s})f_{1}'(y_{s}) > -\frac{(m_{1} - m_{2})^{2}}{4c_{1}c_{2}}$, we have a~stable focus, so when 
partners are more similar in terms of the rate of returning to equilibrium  in solitude, the chance that their 
moods fluctuate at the beginning of the meeting is greater. For $ f_{2}'(x_{s})f_{1}'(y_{s}) < -\frac{(m_{1} - m_{2})^{2}}{4c_{1}c_{2}}$, 
we have a~stable node, and their moods at the meeting consistently approach the equilibrium typical for this 
particular pair. An example of such a behaviour is presented in Fig.~\ref{np1}.

\begin{figure}[!ht]
\centering\includegraphics[height=5cm]{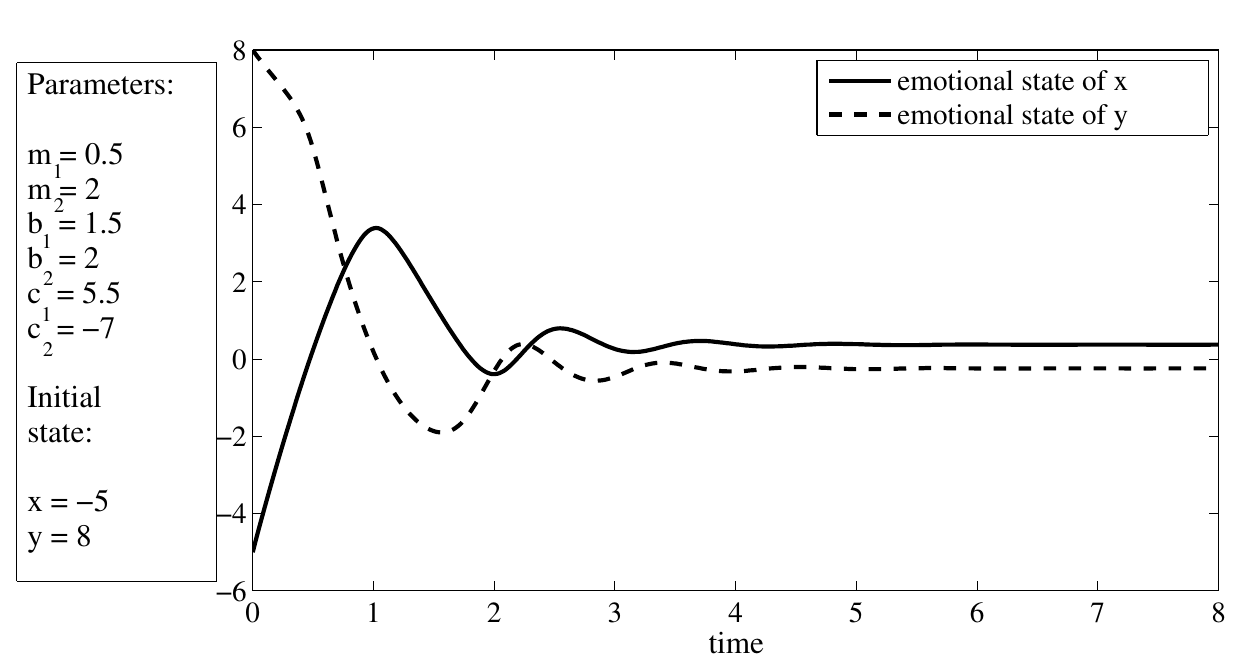}\\
\includegraphics[height=5cm]{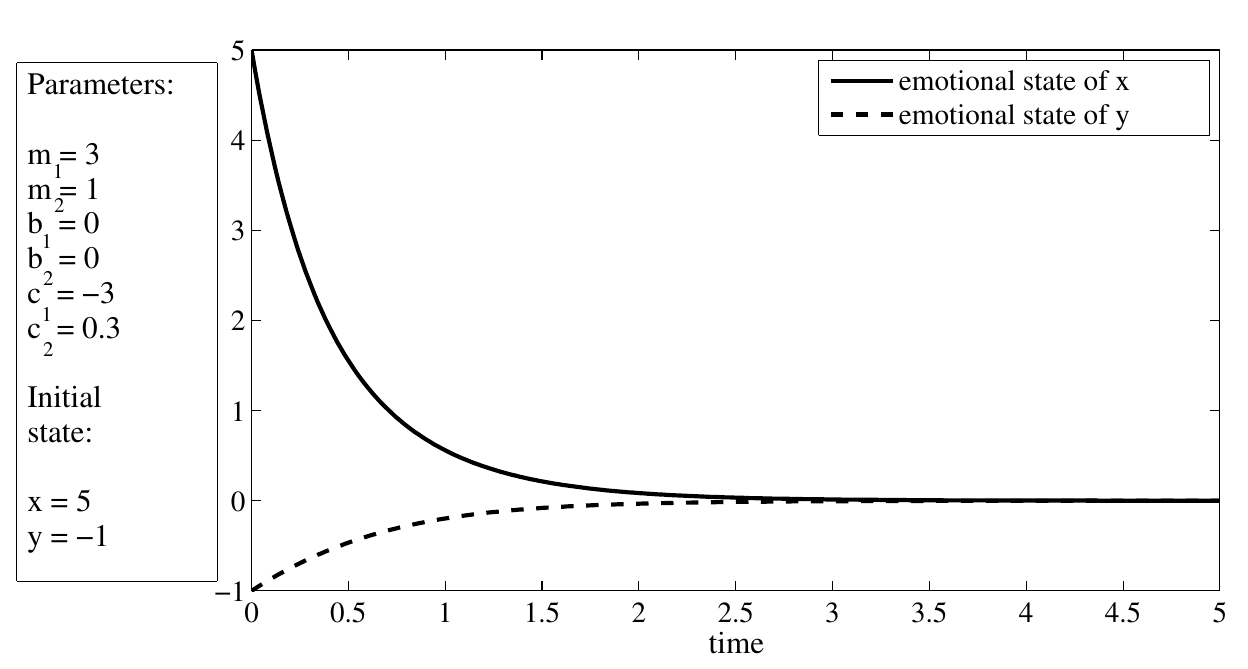}
\caption{Examples of changes in the emotional states of two people with the opposite attitude to each 
other. Top:  the steady state is a~stable focus and emotions oscillate. 
Bottom: the steady state is a stable node and emotions  slowly approach the equilibrium.}\label{np1}
\end{figure}

If the partners have the same attitude to each other, whether it is positive or negative, we get from one to three steady states. For partners having relatively weak effect to each other (\textit{i.e.} $ 0 <c_ {1} c_ {2} <m_ {1} m_ {2} $) or those with strong mutual influence with both being extreme optimists or pessimists, there is only one steady state which is a~stable node.

\begin{figure}[!ht]
\centering
\includegraphics[height=5.0cm]{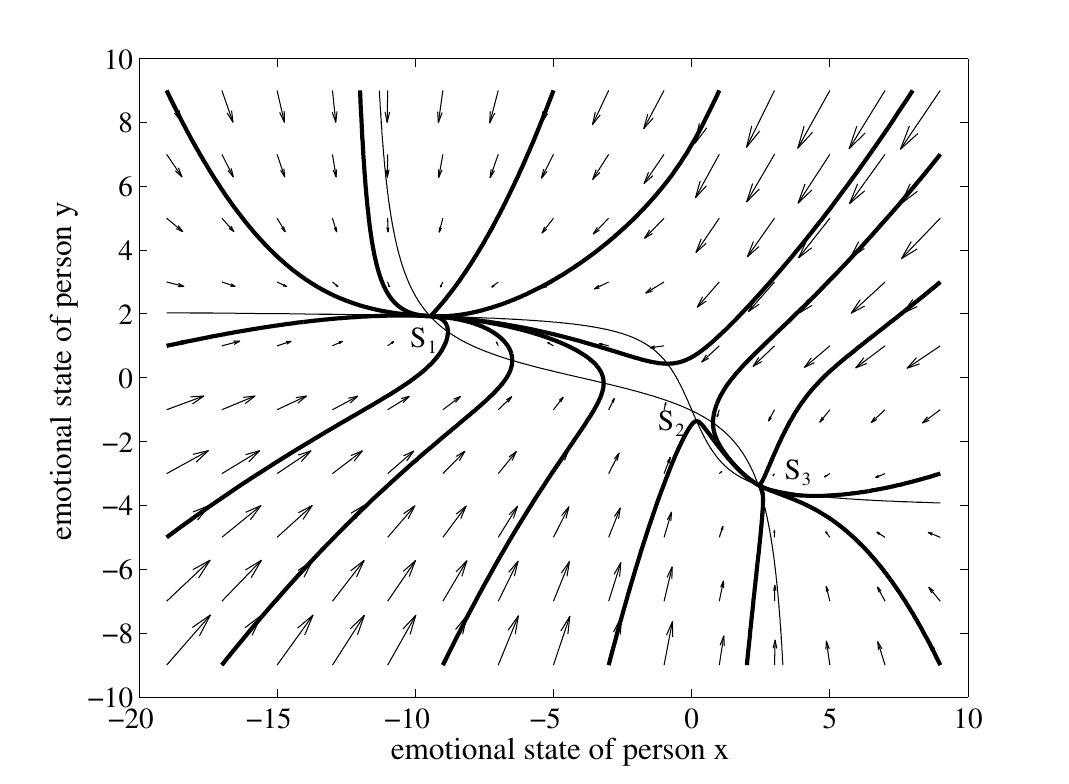}
\includegraphics[height=5.0cm]{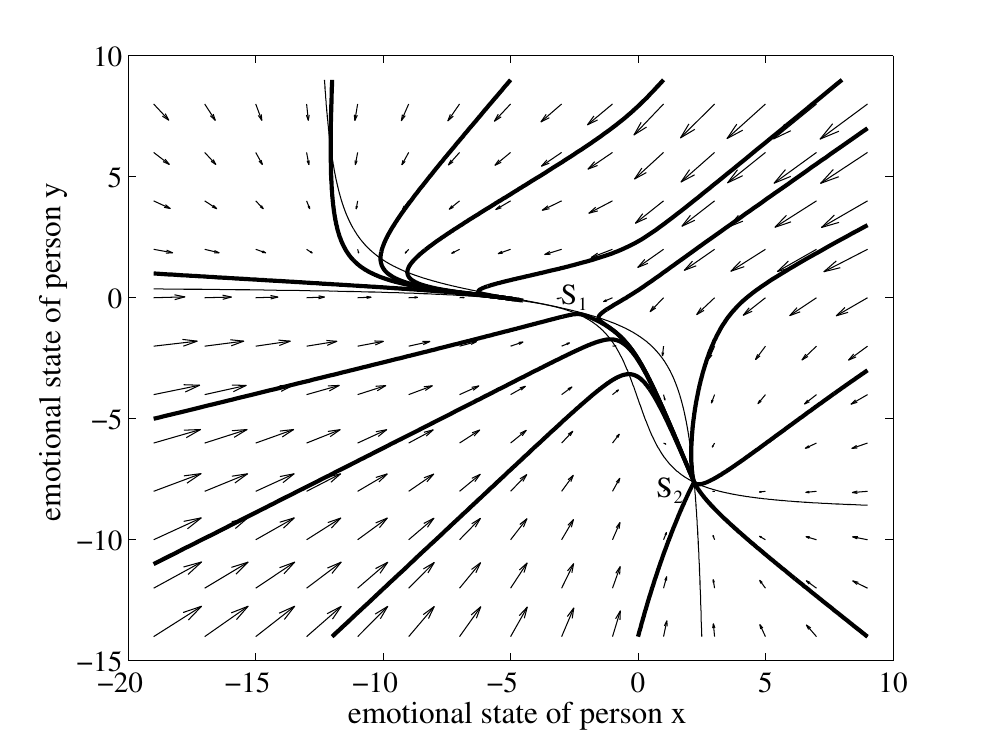}
\caption{Left: phase portrait for System~(\ref{podst_ukl_2}) for $m_{1}=m_{2}=1$, $b_{1}=-5$, 
$b_{2}=-4.19$, $c_{1}=-5$ and $c_{2}=-3$. Right: phase portrait for System~(\ref{podst_ukl_2}) for 
$m_{1}=1$, $m_{2}=2$, $b_{1}=-4$, $b_{2}=-2$, $c_{1}=-5$ and $c_{2}=-4$.}\label{22}
\end{figure}

Different type of the model dynamics appears when the shift of the influence function
by a~$b_i/m_i$ is not large enough. When this shift is reduced, null-clines intersect at three points. 
A~steady state located between two others is a~saddle (see $S_2$ in Fig.~\ref {22} left). 
A generic solution of the system starting from some initial state goes from it and tends to one of the two stable steady states $S_1$ and $S_3$. Such type of the model dynamics is known as bi-stability; c.f.~\cite{Murray2002}.
For two enemies one of the stable steady states has positive coordinate for one person and negative for the second one, and the other steady state -- \textit{vice versa}. For two friends, one stable steady state is positive for both of them, and the other one is negative (result not shown).
 Two steady states appear due to the bifurcation, which is saddle--node in this case  (see Fig.~\ref{22} right). It is possible to achieve both of the steady states, but one of 
them (saddle-node point, that is $S_1$ in Fig.~\ref{22} right) is extremely sensitive to the changes of the model parameters.

\section{Results and their psychological interpretation}\label{sec:disc}
Basing on the analysis of the considered model, we are able to give some conclusions about the influence of 
pessimism and optimism on our social interactions. Clearly, people want to meet if they have a~chance to make 
a~profit from the meeting. Indeed, if one of them is always feeling worse after the meeting than when being 
alone, he will avoid meetings. It happens that one of the friends always initiates the meeting and the other one 
sometimes agrees on it, but very unwillingly. This happens when only one of them has a~profit from it. 
Even less chance of meeting have people who are getting worse humour after than before. 
Below we present the more detailed results of the analysis of probability of maintaining friendship for various 
couples depending on if their nature is similar or not. The results are illustrated by numerical simulations prepared 
using MATLAB with $\arctan$ chosen as influence functions.

\subsection{People with neutral uninfluenced emotional state}
We start from the case of people with neutral uninfluenced steady state which is described by the relation $b_1=b_2=0$. 

\begin{figure}[!ht]
\centering
\includegraphics[height=5.3cm]{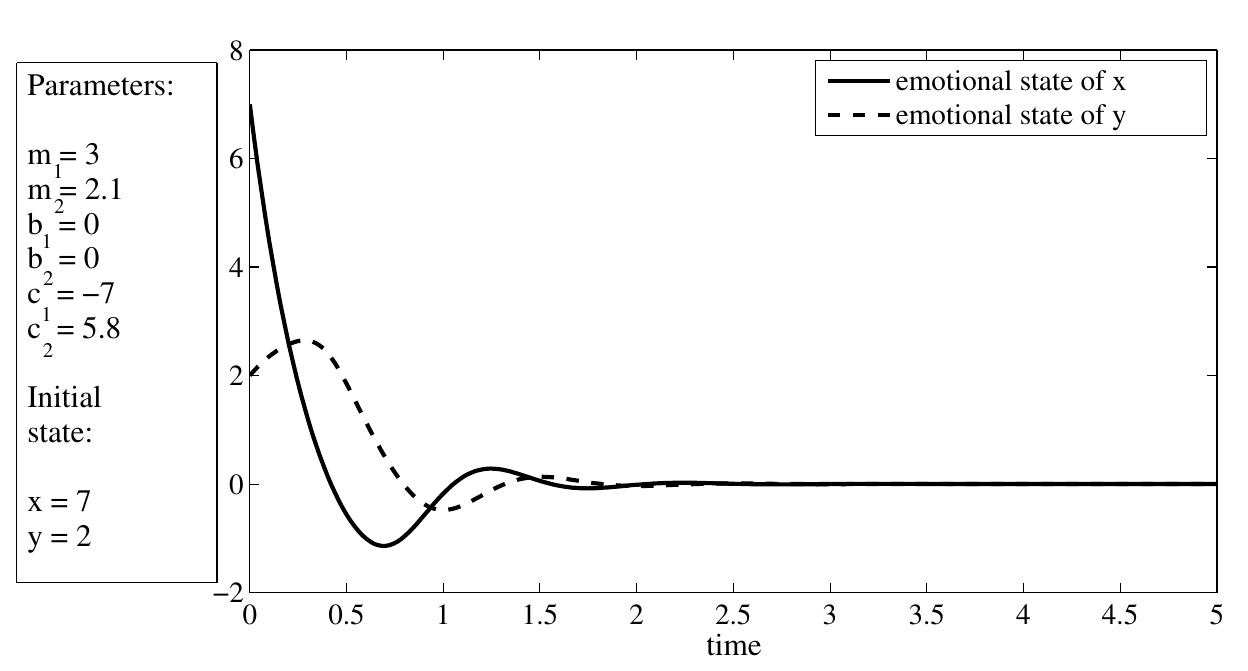}
\caption{ An example of the situation when two persons reach the same emotional state independently if they meet each other  or not.}\label{C.3}
\end{figure}

Generally, if a person like being alone, his/her solitude leads to apathy, and 
when two such persons meet, then the meeting does not influence their emotional states much, especially when their ability to calm emotions  is greater  than  the strength of influence of the partner. During the meeting of such persons, despite they are friends or enemies, their emotional states go monotonically to $0$. It can happen slower or faster than in solitude, but the final result is always the same. However, for persons with different attitudes to each other, there could be oscillatory behaviour, which is associated with the inequality\begin{equation}\label{(5.1)}
\sqrt{|c_1c_2|} > \frac{|m_1-m_2|}{2} .
\end{equation}
We see that when the partners are more similar to each other and their influence each other more, then the chance of oscillatory behaviour is greater, cf.~Fig.~\ref{C.3}, where the example of such  oscillations is presented. 

\begin{figure}[!ht]
\centering
\includegraphics[height=5.25cm]{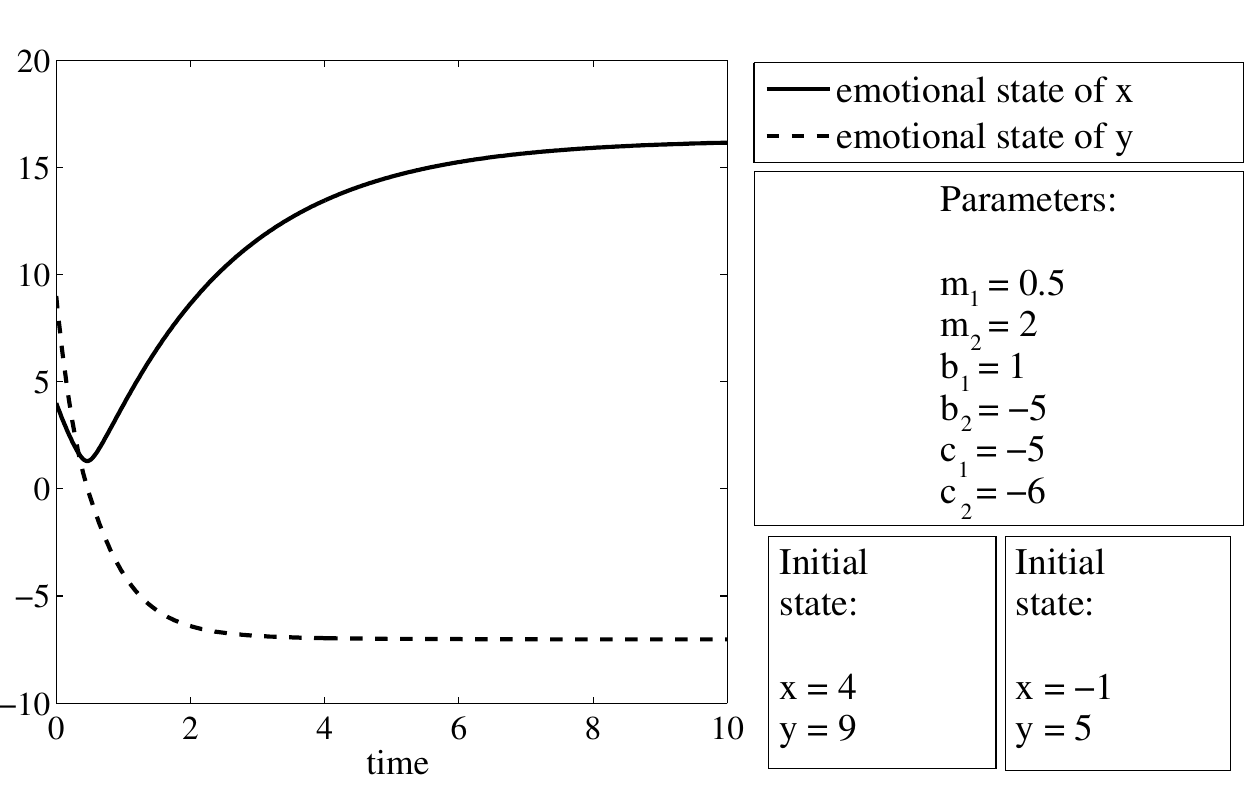}
\includegraphics[height=5.25cm]{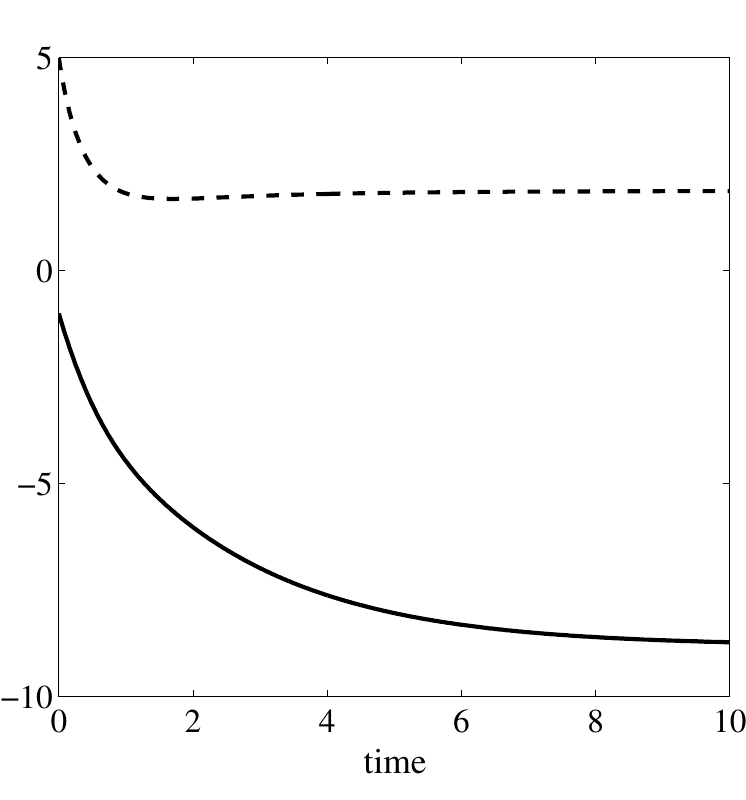}
\caption{En example of two persons for which different initial conditions yield different opposite final emotional states.}\label{B.1}
\end{figure}

Notice, that the case $ m_1 = m_2$ is specific, as oscillations appear 
independently of the strength of influence. However, there are situations when the emotional states of 
such people change much during the meeting. It can happen for two friends or enemies, and if their are sufficiently 
interested in the partner's emotions  ($c_1c_2 > m_1m_2$), then depending on the initial data they could achieve 
some level of happiness or dissatisfaction. Two positively oriented persons reciprocate their each other emotions, 
while two negatively oriented persons have opposite emotions, which is visible on the phase portrait presented in 
Fig.~\ref{22} and two exemplary graphs of the emotional state in Fig.~\ref{B.1}. 
Moreover, in Fig.~\ref{B.1} we see that partners with stronger influence and weaker forgetting, stronger experience contacts with others comparing to persons weakly influenced. Although the graphs presented in Fig.~\ref{B.1} do not reflect the emotions of neutrally oriented persons, but their correctly 
reproduce the relations described above. 
What is interesting, in the presence of other people, emotions of the considered person could be enlarged, but on the other hand too high emotions could be repressed. 
An example of the situation when  the emotions change is presented in Fig.~\ref{A.2}. It reflects the meeting of two friends with one of them having a good mood and the other having bad one. Fig.~\ref{A.2} shows that at the beginning the emotional state of a sad person getting better, while the state of his friend is getting worse, but eventually  their overcome the bad mood and start to be pleased together. 

\begin{figure}[!ht]
\centering
\includegraphics[height=5cm]{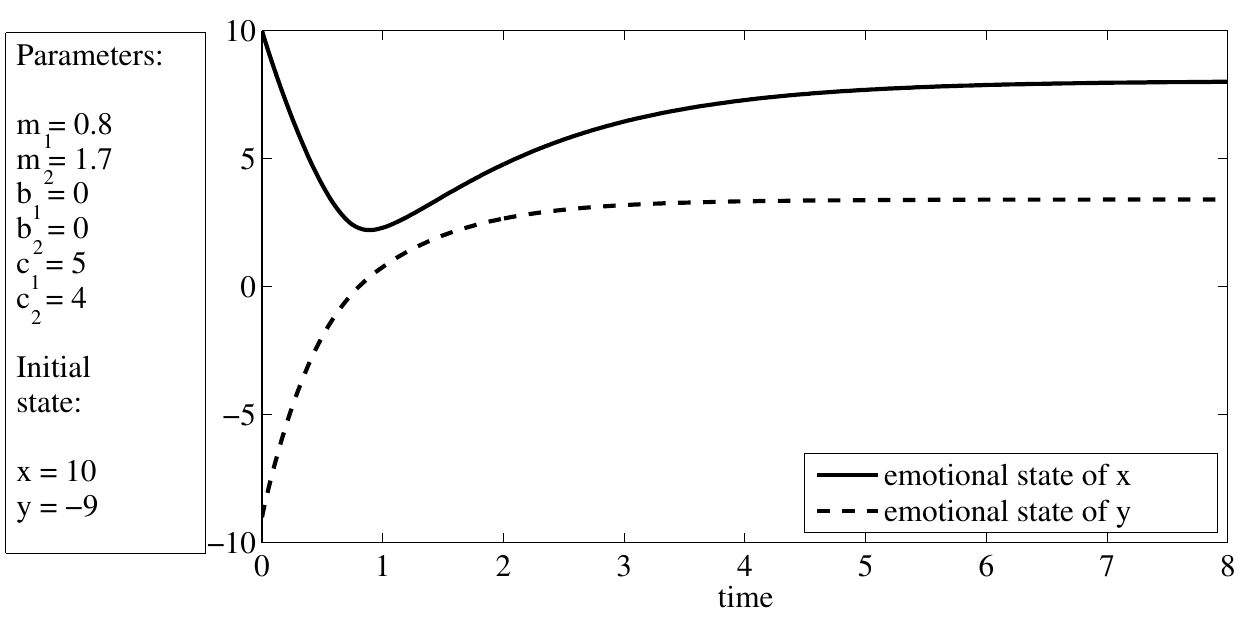}
\caption{An example of the change of emotional states of two friends with the opposite initial states.}\label{A.2}
\end{figure}

\begin{figure}[!ht]
\centering
\includegraphics[height=5cm]{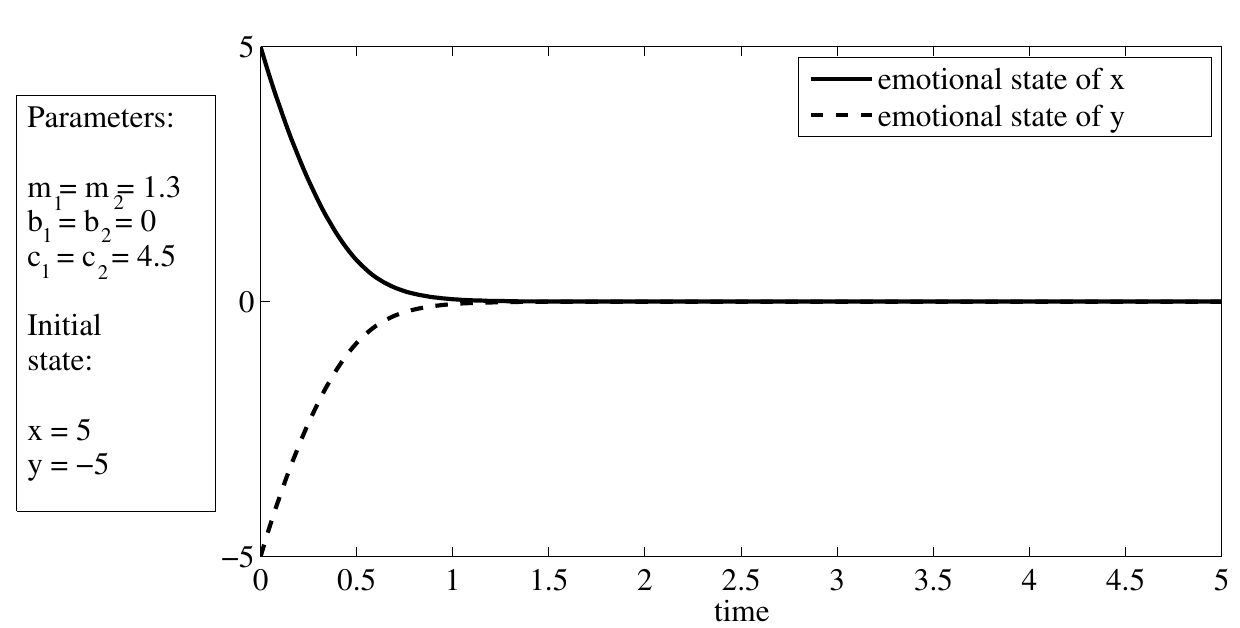}
\caption{Emotional states of two strongly dependent on each other friends with the same parameters describing both of them but with the opposite initial states.}\label{A.3}
\end{figure}

It can also happen that two friends having initially opposite emotions approach the neutral state after some time. 
From Section~\ref{sec:res} we know that their emotional states should tend to one of the stable steady states, for which both coefficients have the same sign. 
However, if the initial state lies within the stable manifold of the saddle,   
then the solution goes to this saddle. When the friends have the same characteristics  ($m_1 = m_2$,  $c_1 = c_2$,  $f_1 = f_2$), then their emotional states in solitude go to neutral state ($b_1 = b_2 = 0$) and the impact of negative and positive emotions is the same. In such a case the stable manifold is described as  the straight line  $y = -x$. Therefore, if the partners have initially  opposite emotional states, then they react in such a way their influence of each other is weak. This situation is reflected in Fig.~\ref{A.3}. 

Exactly the same result could be obtained for two enemies having initially the same emotions. It is obvious that in both cases the partners approach neutral steady state.  

\subsection{Change of dynamics initiated by one of the partners}
As we have mentioned above, two enemies can calm their emotions or one of the partners can enjoy the misfortune of the another. Clearly, this situation is comfortable only for the first person. The second person have two ways out. Firstly, he/she can finish the meeting, secondly, he can reverse the situation,  however the second possibility could be achieved only under some circumstances.
This person should  pretend to be a friend, and if Condition~\eqref{(5.1)} is satisfied, then oscillations of emotions will lead to neutral state. 
However, if the ``pro temporary friend'' comes back to his real state, then the final emotional state is the reverse of the initial one. The proper moment to change the behaviour from temporary friendship to hostility  depends on the model parameters. 

\begin{figure}[!ht]
\centering
\includegraphics[height=5cm]{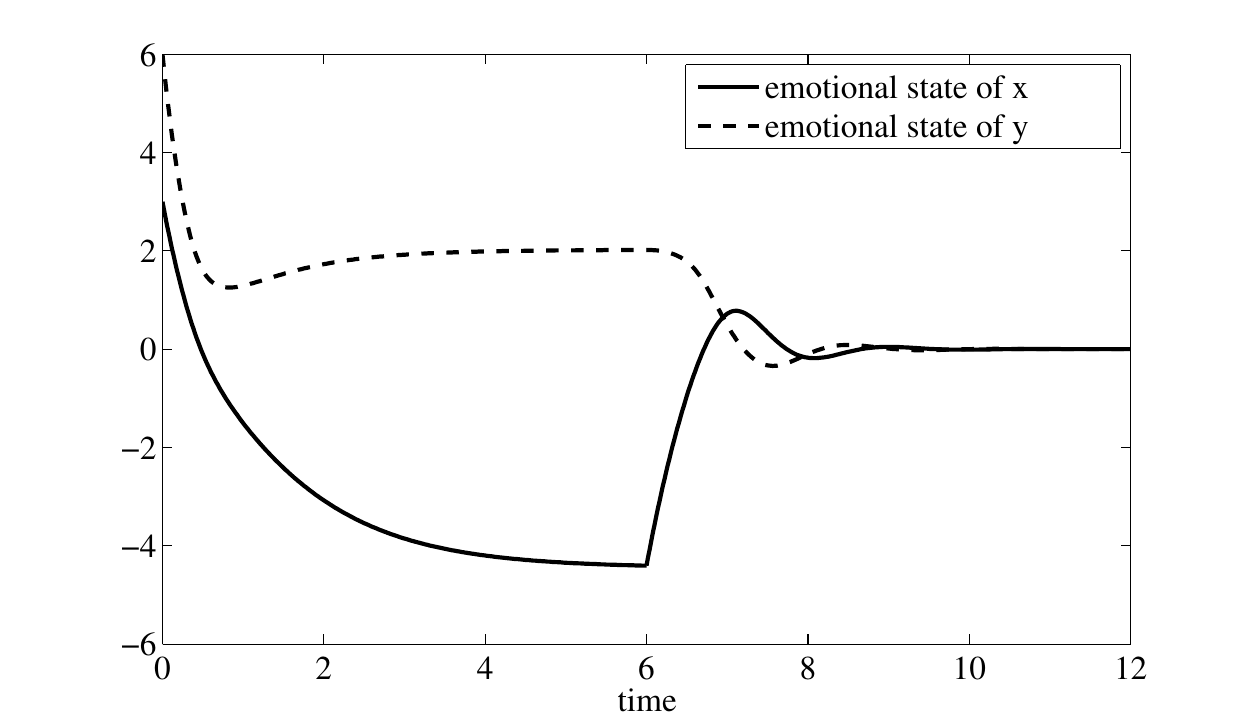}
\includegraphics[height=5cm]{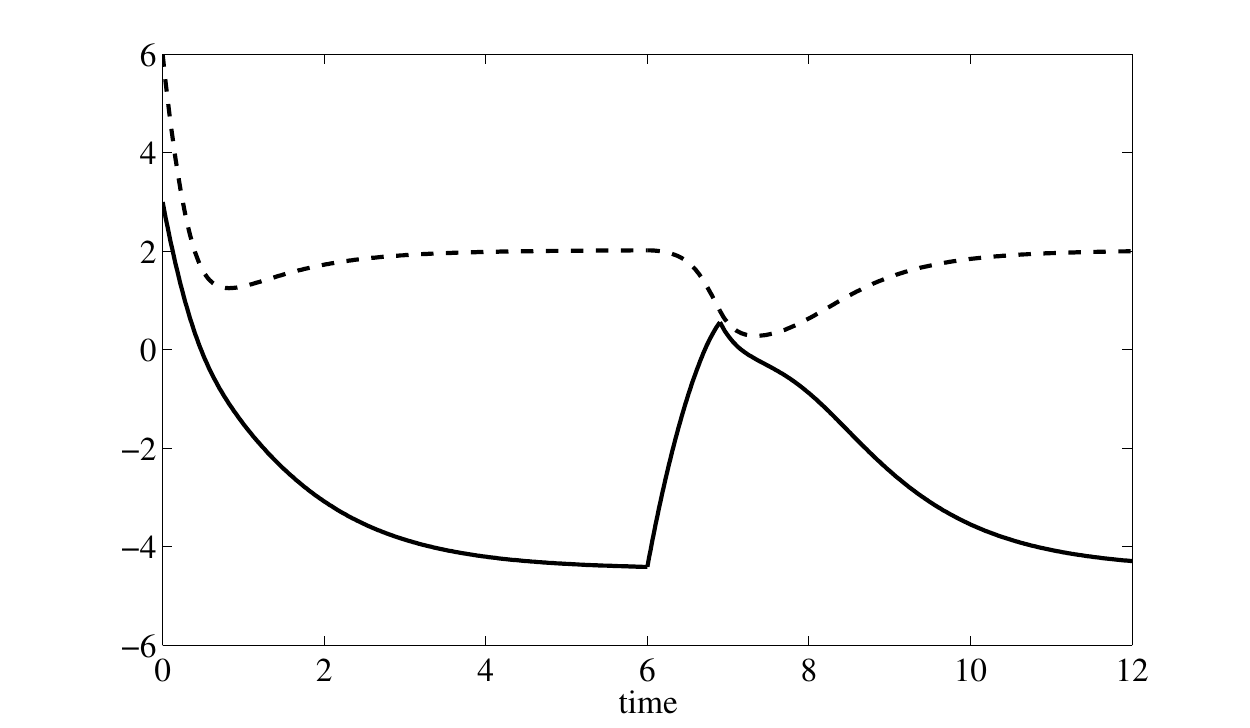}
\includegraphics[height=5cm]{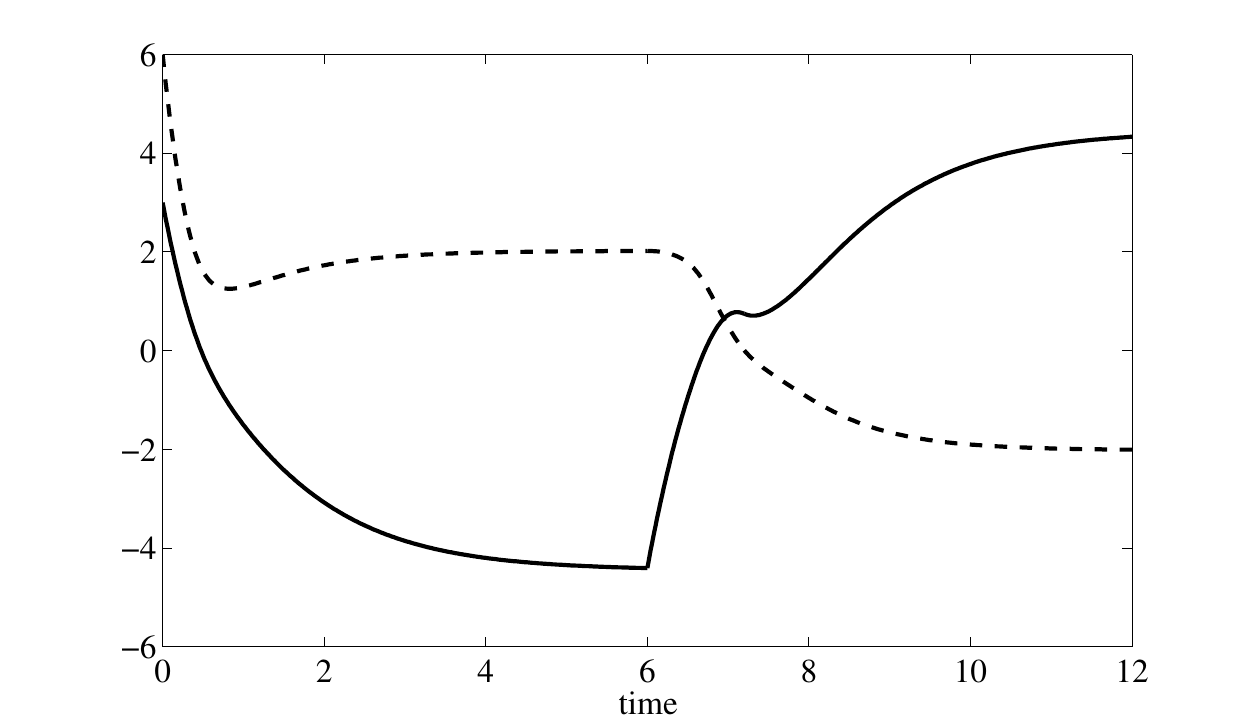}
\caption{Possible changes of the dynamics of the emotional states initiated by one of the interacting partners.}\label{5.1}
\end{figure}

The possible changes of the emotional states in time are presented in Fig.~\ref{5.1}.

The  graphs show the changes  of the emotional states for two enemies. Top graph shows the solution of Eqs.~\eqref{podst_ukl_2} till $t = 6$, and  at  $t = 6$ the first person 
changes his attitude from negative to positive of the same strength. 
This leads to improving of the   emotional  state of the person $ x $ and 
decreasing of satisfaction of the person  $y$ to $0$. 
The middle graph shows the unsuccessful attempt of the change of the emotional state of the person $x$. This attempt is unsuccessful due to the early change of the attitude back to negative.
 The bottom graph shows
the situation in which the person changes his attitude to negative again at $ t = 7$, and this leads to increase his positive emotions together with negative emotions of the partner.  It should be noticed that even if Condition~\eqref{(5.1)} is not satisfied, the unhappy partner should change his attitude to the other (being his enemy), as both partners calm their emotions in such a case, as presented in the bottom graph in~Fig.~\ref{np1}. 


The line of reasoning presented above is able to explain so-called Stockholm syndrom, when kidnapped person 
starts to feel positive emotions for the kidnapper. This is just a smart defence mechanism, which is able to 
calm emotions of both the kidnapped  and kidnapper, and decrease the dangerous of kidnapper.

\subsection{Relationship of a negatively oriented pessimist}  
As it turns out, a~pessimist negatively oriented towards the partner may sustain the relationship with another pessimist who is positively oriented towards him. He may also be in relationship with a~pessimist similar to him whose reciprocal attitude is also negative, but having a~strong influence on each other, the pessimist will still have to ``take risk''.
In result, sometimes he will achieve positive and sometimes negative steady state. Favourable situation for both pessimists is shown in Fig.~\ref{nn2}, where the thin dotted straight line illustrates the steady state which is achieved by each of the enemies when spending time alone. Curves representing the emotional states of these two people are above the dotted line after some time, so they have benefits from the meeting.

\begin{figure}[!ht]
\centering
\includegraphics[height=5.3cm]{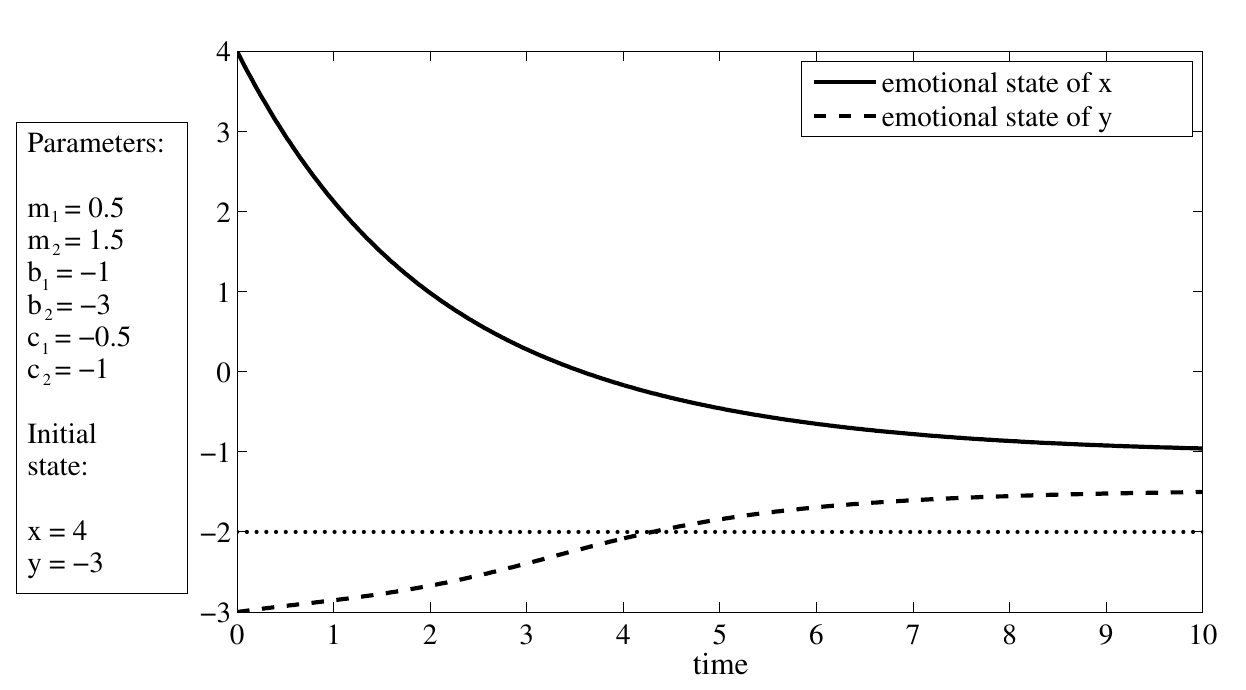}
\caption{
Graph of the emotional states of two pessimists who are enemies. Signs of the initial conditions are opposite.
}\label{nn2}
\end{figure}

\subsection{Relationship of a positively oriented pessimist}
When a~pessimist has a~positive attitude to his partner, he will have the chance to form a~relationship with a~moderate optimist, a~pessimist who does not like him, or an optimist who likes him. The more extreme optimist his partner would be, the better result could be achieved. This is the best company for a~pessimist. The only problem may be that optimists have more profit from contacts with other optimists, so meeting with a~pessimist may be disadvantageous, despite the seeming gain. 

\subsection{Relationship of an optimist}
For an optimist, it is better to spend the time on meeting with other optimists.
If the optimist has a~negative attitude towards the partner, he can sustain a~relationship with a~pessimist who likes him (if the severity of optimism and pessimism is similar) or alternately with a~partner who does not like him and is slightly optimistic or pessimistic. The strength of their influence on each other must be so large that there were two stable steady states. However, an optimist most willingly will meet with friends who think positively about life, as he is.


\section{Conclusions}\label{sec:conc}
From the model analysis several conclusions appear. 
First, only persons with neutral uninfluenced steady state are able to feel similarly in solitude and being with the partner. 
Second, for two enemies, it is enough that one of the partners changes his/her behaviour to obtain complete change of the emotional states of both of them. 
Next, pessimists, which is not surprising, may have greater 
difficulties with finding a~partner. However, they may have more varied contacts than optimists. Optimists can 
take advantages almost only of the mutual friendship and they should avoid people with negative attitude,  
whereas pessimists may like other people or not and they will reap the benefits from that and even though 
they have more opportunities than optimists, their contacts are less satisfactory than optimists contacts. 
Moreover, it happens that in order to make a~profit  they must have better mood than a~partner before the meeting. This situation is illustrated in Fig.~\ref{B.1}. 
The graphs show the relation between the emotional 
states of two partners negatively oriented
towards each other. The differences in these graphs are only due to different initial states of partners. In the 
first graph the partners approach  the steady state favourable to the person $ x $ being an optimist.
 This is despite the better initial mood of the person $ y $ being a~pessimist. In the second case, pessimism of 
the person $ y $ causes that the person ``has won'' his emotional state not too highly.

\begin{figure}[!ht]
\centering
\includegraphics[height=5.8cm]{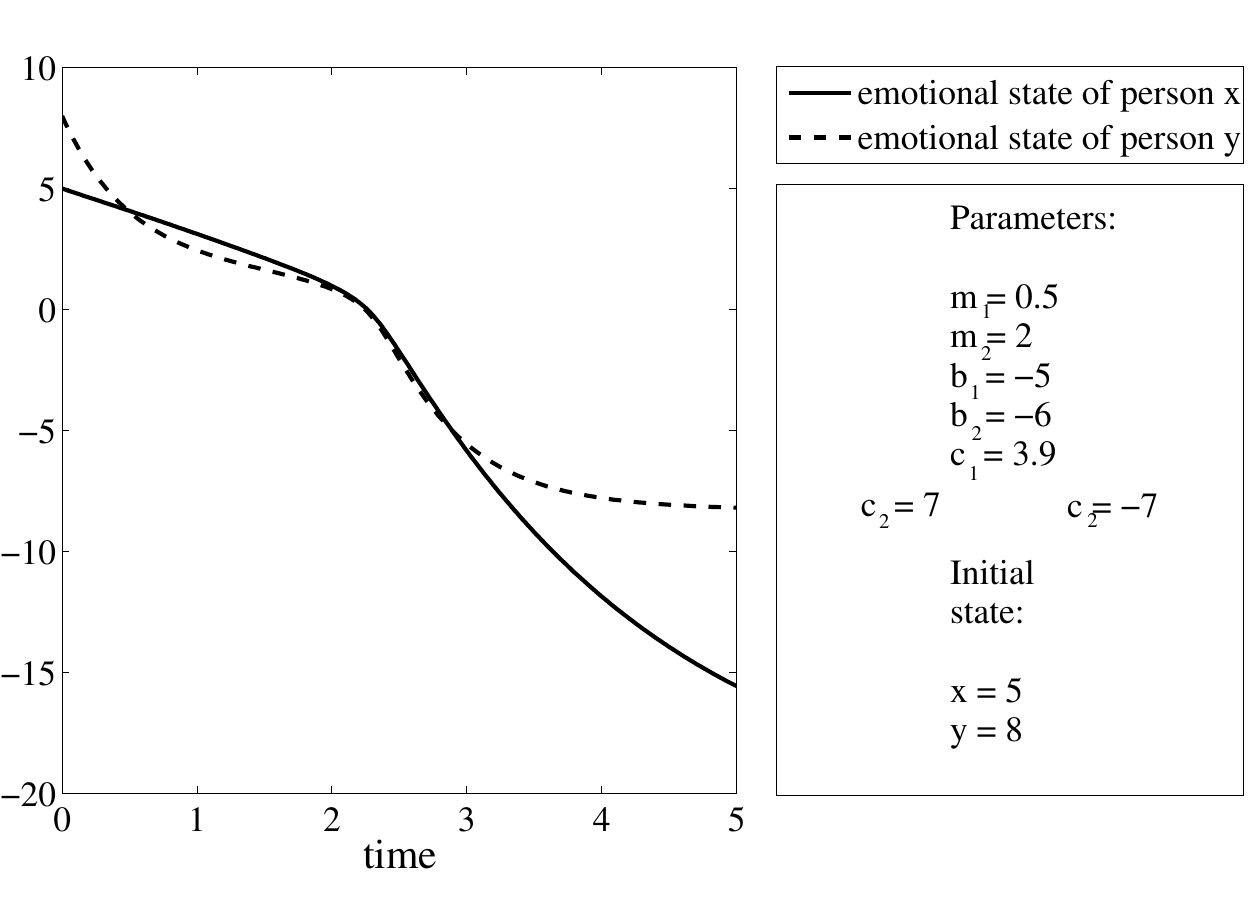}
\includegraphics[height=5.8cm]{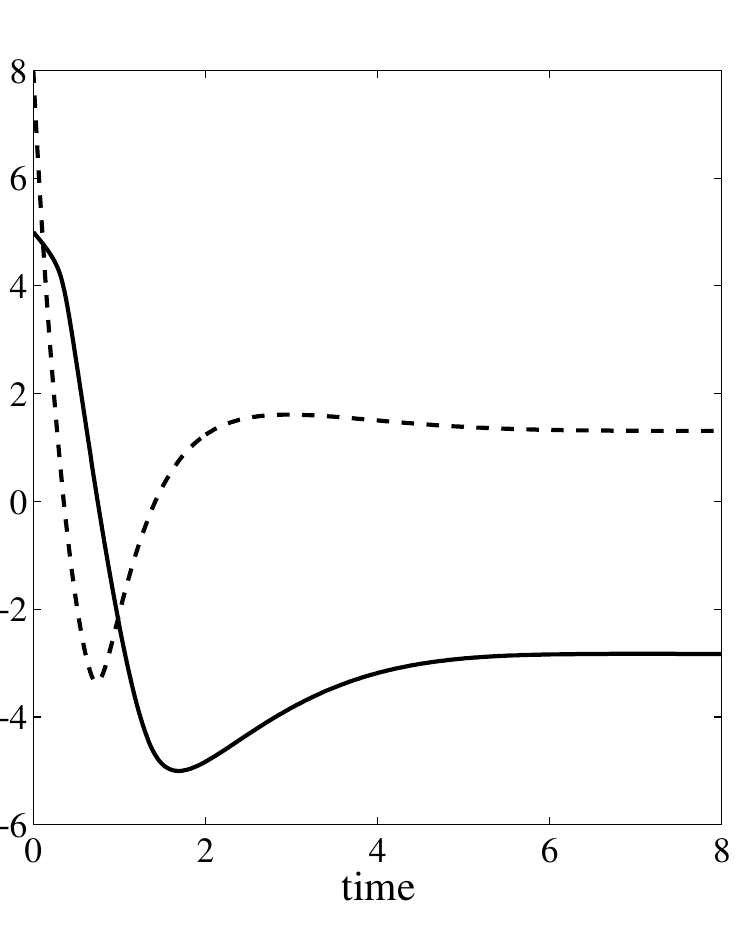}
\caption{ Comparison of the dynamics of the emotional states of two pessimists when they are friends ($c_{1},c_{2} > 0$) and when they have inconsistent relations ($c_{1}>0$, $c_{2}<0$).}\label{pp1}
\end{figure}

An interesting fact which has stemmed from the analysis is that the pessimistic enemies may 
be not friendly to each other
 and have it enhance their moods, while pessimistic friends will mutually worsen their moods. Fig.~\ref {pp1} 
perfectly illustrates an example in which for two pessimists it is better to have a~different relationship to each 
other than friendship. The first graph shows 
that a~meeting of two friends always causes they approach an unprofitable state due to the strong pessimistic 
tendency of them both. This state is more negative than they could achieve being alone. As shown in the 
second graph, it would be advantageous for both of them to change attitudes to the partner who is more 
pessimistic. Then they both could have gained from the meeting.

At the end we should mark that people do not like to feel diametrically changing emotions, and therefore almost all people prefer to interact with partners having the same attitude. It particularly considers optimistic persons. 

The model presented in this paper is very simple but clear, and much more conclusions could be drawn from it. As we have noticed at Introduction and Model Description sections, it gives a possibility of different interpretations, such that we expect it could be farther exploited in the future. 

\section{Acknowledgement}
The part of results presented in this paper has been presented during the XIX National Conference of Applications of Mathematics in Biology and Medicine, Jastrz\c{e}bia G\'{o}ra, September 16–20, 2013~\cite{ziolkowka}.
 \medskip















\end{document}